\documentclass{amsart}
\usepackage{graphicx} 

\title{A note on ubiquity of geometric Brascamp--Lieb data}
\author{Neal Bez}
\address[Neal Bez]{Department of Mathematics, Graduate School of Science and Engineering,
Saitama University, Saitama 338-8570, Japan}
\email{nealbez@mail.saitama-u.ac.jp}

\author{Anthony Gauvan}
\address[Anthony Gauvan]{Department of Mathematics, Graduate School of Science and Engineering,
Saitama University, Saitama 338-8570, Japan}
\email{anthonygauvan@mail.saitama-u.ac.jp}

\author{Hiroshi Tsuji}
\address[Hiroshi Tsuji]{Department of Mathematics, Graduate School of Science and Engineering,
Saitama University, Saitama 338-8570, Japan}
\email{tsujihiroshi@mail.saitama-u.ac.jp}

\usepackage{amsfonts, amssymb, amsmath, verbatim, amsthm, mathrsfs, amscd, enumerate, ascmac, fancyhdr, caption, hyperref, framed, subfigure}
\usepackage{bbm}
\numberwithin{equation}{section}
\usepackage{tikz}
\usetikzlibrary{positioning,intersections, calc, arrows,decorations.markings, angles}

\newtheorem{theorem}{Theorem}[section]
\newtheorem{lemma}[theorem]{Lemma}
\newtheorem*{lemma*}{Lemma}
\newtheorem{proposition}[theorem]{Proposition}
\newtheorem*{proposition*}{Proposition}
\newtheorem{corollary}[theorem]{Corollary}

\newtheorem{question}[theorem]{Question}
\theoremstyle{definition}

\newtheorem*{claim*}{Claim}

\theoremstyle{remark}
\newtheorem*{remark}{Remark}

\thanks{
This work was supported by JSPS Kakenhi grant numbers 22H00098, 23K25777, 24H00024  (Bez), 
23KF0188 (Gauvan), and 24KJ0030 (Tsuji).
}

\begin{document}

\maketitle

\begin{abstract}
Relying substantially on work of Garg, Gurvits, Oliveira and Wigderson, it is shown that geometric Brascamp--Lieb data are, in a certain sense, ubiquitous. This addresses a question raised by Bennett and Tao in their recent work on the adjoint Brascamp--Lieb inequality.
\end{abstract}


\section{Background}
Given linear transformations $B_j : \mathbb{R}^n \to \mathbb{R}^{n_j}$ and exponents $c_j > 0$ for $j \in [m]$, consider the inequality
\begin{equation} \label{e:BL}
    \int_{\mathbb{R}^n} \prod_{j=1}^m f_j(B_jx)^{c_j} \, dx \leq C \prod_{j=1}^m \bigg(\int_{\mathbb{R}^{n_j}} f_j \bigg)^{c_j}
\end{equation}
for non-negative and integrable functions $f_j : \mathbb{R}^{n_j} \to \mathbb{R}$. H\"older's inequality, the Loomis--Whitney inequality and Young's convolution inequality are important special cases, but the systematic study of \eqref{e:BL} in the above framework began in earnest in work of Brascamp and Lieb \cite{BL} and as a result inequality \eqref{e:BL} is known as the \emph{Brascamp--Lieb inequality}. 

Following standard notation, we write $\mathbf{B} = (B_j)_{j=1}^m$, $\mathbf{c} = (c_j)_{j=1}^m$, and call them $m$-transformations and $m$-exponents, respectively. The pair $(\mathbf{B},\mathbf{c})$ is called a \emph{Brascamp--Lieb datum}. The best (smallest) constant in \eqref{e:BL} is called the \emph{Brascamp--Lieb constant} and given by
\[
\mathrm{BL}(\mathbf{B},\mathbf{c}) = \sup_{\mathbf{f}} \mathrm{BL}(\mathbf{B},\mathbf{c};\mathbf{f}),
\]
where
\[
\mathrm{BL}(\mathbf{B},\mathbf{c};\mathbf{f}) = \frac{\int_{\mathbb{R}^n} \prod_{j=1}^m f_j(B_jx)^{c_j} \, dx}{\prod_{j=1}^m (\int_{\mathbb{R}^{n_j}} f_j )^{c_j}}
\]
and the supremum is taken over all inputs $\mathbf{f} = (f_j)_{j=1}^m$ such that $\int_{\mathbb{R}^{n_j}} f_j \in (0,\infty)$. If the supremum is attained by some input $\mathbf{f}$, then the datum $(\mathbf{B},\mathbf{c})$ is said to be \emph{extremisable}. In the case $\mathrm{BL}(\mathbf{B},\mathbf{c}) < \infty$, we say that the datum $(\mathbf{B},\mathbf{c})$ is \emph{feasible} (irrespective of whether it is extremisable or not).

A remarkable theorem of Lieb says that  the Brascamp--Lieb constant is exhausted by centred gaussians. To state Lieb's theorem precisely, let us also introduce the notation
\[
\mathrm{BL_g}(\mathbf{B},\mathbf{c}) = \sup_{\mathbf{A}}\mathrm{BL}(\mathbf{L},\mathbf{c};\mathbf{A}).
\]
Here we are slightly misusing notation and writing
\[
\mathrm{BL}(\mathbf{B},\mathbf{c};\mathbf{A}) := \mathrm{BL}(\mathbf{B},\mathbf{c};\mathbf{f})
\]
when $f_j(x) = \exp(-\pi \langle A_jx,x\rangle)$ for some positive definite transformation $A_j : \mathbb{R}^{n_j} \to \mathbb{R}^{n_j}$.
\begin{theorem}\cite{Lieb} \label{t:Lieb}
For any Brascamp--Lieb datum $(\mathbf{B},\mathbf{c})$ we have
\[
\mathrm{BL}(\mathbf{B},\mathbf{c}) = \sup_{\mathbf{A}} \mathrm{BL}(\mathbf{B},\mathbf{c};\mathbf{A}).
\]
\end{theorem}

In addition to the special cases of the Brascamp--Lieb inequality mentioned already, we next introduce the geometric Brascamp--Lieb inequality. The Brascamp--Lieb datum $(\mathbf{G},\mathbf{c})$ is said
to be \emph{geometric} if 
\begin{equation} \label{e:projection}
G_jG_j^* = I_{n_j} \qquad (j \in [m])    
\end{equation}
and
\begin{equation} \label{e:isotropy}
\sum_{j=1}^m c_j G_j^*G_j = I_{n}.
\end{equation}
Of course, \eqref{e:projection} simply means that each $G_j^*$ is an isometry, whilst \eqref{e:isotropy} can be viewed as a kind of transversality condition involving the orthogonal projections $G_j^*G_j$. For later use, it will be convenient to introduce the class of $m$-transformations
\[
\mathcal{G}(\mathbf{c}) := \{ \mathbf{G} : \text{$(\mathbf{G},\mathbf{c})$ is geometric} \}
\]
associated with a fixed $m$-exponent $\mathbf{c}$.
\begin{theorem} \cite{Ball_GBL, Barthe_Invent} \label{t:GBL}
    If $\mathbf{G} \in \mathcal{G}(\mathbf{c})$ then $\mathrm{BL}(\mathbf{G},\mathbf{c}) = 1$. In fact,
    \[
    \mathrm{BL}(\mathbf{G},\mathbf{c}) = \mathrm{BL}(\mathbf{G},\mathbf{c};\mathbf{A})
    \]
    with $A_j = I_{n_j}$, $j \in [m]$.
\end{theorem}
The geometric Brascamp--Lieb inequality was discovered first by Ball \cite{Ball_GBL} when $n_j = 1$ for all $j \in [m]$, and later Barthe \cite{Barthe_Invent} extended Ball's result to higher dimensions. Ball was able to establish some remarkable results in convex geometry using the geometric Brascamp--Lieb inequality, including progress on the cube slicing problem (see, for example, \cite{Ball_GBL}). Perhaps surprisingly, it turns out that geometric Brascamp--Lieb data play a fundamental role in the general theory of the Brascamp--Lieb inequality. In particular, geometric data are often referred to in terms like ``generic". One manifestation of this is the following fundamental result of Bennett, Carbery, Christ and Tao \cite{BCCT}.
\begin{theorem} \cite{BCCT} \label{t:BCCT_geometricequivalent}
The Brascamp--Lieb datum $(\mathbf{B},\mathbf{c})$ is extremisable if and only if $(\mathbf{B},\mathbf{c})$ is equivalent to a geometric datum.
\end{theorem}
The notion of \emph{equivalent data} is as follows. As in \cite[Definition 3.1]{BCCT}, we say that $\mathbf{B}$ and $\widetilde{\mathbf{B}}$ are equivalent if 
\begin{equation*}
\widetilde{B}_j = T_j^{-1}B_jT
\end{equation*}
for invertible transformations $T_j : \mathbb{R}^{n_j} \to \mathbb{R}^{n_j}$ and $T : \mathbb{R}^n \to \mathbb{R}^n$, called \emph{intertwining transformations}. This defines an equivalence relation and we write 
$
[\mathbf{B}]
$
for the equivalence class containing $\mathbf{B}$. The data $(\mathbf{B},\mathbf{c})$ and $(\widetilde{\mathbf{B}},\widetilde{\mathbf{c}})$ are equivalent if $\mathbf{B}$ and $\widetilde{\mathbf{B}}$ are equivalent and $\mathbf{c} = \widetilde{\mathbf{c}}$. By linear changes of variables on both sides of \eqref{e:BL}, it is straightforward to see  that
\begin{equation} \label{e:equivdataBLconstants}
    \mathrm{BL}(\widetilde{\mathbf{B}},\mathbf{c}) = \frac{\prod_{j=1}^m |\det T_j|^{c_j}}{|\det T|} \mathrm{BL}(\mathbf{B},\mathbf{c})
\end{equation}
when $\widetilde{\mathbf{B}} \in [\mathbf{B}]$ with intertwining transformations $T_j : \mathbb{R}^{n_j} \to \mathbb{R}^{n_j}$ and $T : \mathbb{R}^n \to \mathbb{R}^n$; see \cite[Lemma 3.3]{BCCT}. Although it is completely obvious from \eqref{e:equivdataBLconstants}, we would like to emphasise that finiteness of the Brascamp--Lieb constant is preserved under equivalence.

Some of the consensus regarding the genericity of geometric data presumably stems from Theorem \ref{t:BCCT_geometricequivalent} and the expectation that the set of extremisable data should be, in some sense, big. In a recent paper by Bennett and Tao \cite{BennettTao}, in which the \emph{adjoint Brascamp--Lieb inequality} was first discovered, the question of whether extremisable data are in fact dense in the space of feasible data was raised; see \cite[Question 10.6]{BennettTao}. A key estimate in \cite[Theorem 1.11]{BennettTao} is easily verified for extremisable data and so the density of such data, if true, could conceivably be combined with certain continuity arguments to extend the estimate to all data.

It turns out that although extremisable data are not dense in the space of feasible data (for somewhat uninteresting reasons\footnote{For example, see the Remark after Theorem \ref{t:genericity}.}), we shall see that one has density ``modulo equivalence of data". In other words, we shall show that given any feasible datum, an element of its equivalence class can be arbitrarily well approximated by a geometric datum. This is what we refer to as the ``ubiquity" of geometric data and provides a reasonable answer to \cite[Question 10.6]{BennettTao}. In the next section we present such a result; we would like to emphasise that we rely substantially on arguments of Garg, Gurvits, Oliveira and Wigderson \cite{GGOW} (see the Remark after the proof of Proposition \ref{p:gconv} for further clarification). After that, in Section \ref{section:ABL}, we shall see how the ubiquity of geometric data is sufficient to prove the aforementioned key estimate of Bennett--Tao in \cite[Theorem 1.11]{BennettTao} along the lines alluded to above.

\section{Ubiquity of geometric data} \label{section:ubiquity}

 A very useful observation is that if $(\mathbf{B},\mathbf{c})$ is feasible, then it can be transformed to an equivalent data which satisfies \emph{either} one of the two conditions defining geometric data, \eqref{e:projection} and \eqref{e:isotropy}. 
\begin{proposition} \label{p:normalisation}
Suppose $\mathrm{BL}(\mathbf{B},\mathbf{c}) < \infty$.
\begin{itemize}
    \item[(1)] Set 
\begin{equation*}
    \phi(\mathbf{B})_j := M_j^{-1/2}B_j, \qquad M_j := B_jB_j^*.
\end{equation*}
Then $\phi(\mathbf{B}) \in [\mathbf{B}]$ and $\phi(\mathbf{B})_j\phi(\mathbf{B})_j^* = I_{n_j}$ for all $j \in [m]$.
   \item[(2)] Set
\begin{equation*}
    \psi(\mathbf{B})_j := B_jM^{-1/2}, \qquad M := \sum_{j=1}^m c_jB_j^*B_j.
\end{equation*}
Then $\psi(\mathbf{B}) \in [\mathbf{B}]$ and $\sum_{j=1}^m c_j \psi(\mathbf{B})_j^*\psi(\mathbf{B})_j = I_{n}$.
\end{itemize}
\end{proposition}
\begin{proof}
For (1), it is immediate that $\phi(\mathbf{B}) \in [\mathbf{B}]$ and $\phi(\mathbf{B})_j\phi(\mathbf{B})_j^* = I_{n_j}$ holds, as long as we justify the existence of $M_j^{-1/2}$. This follows since $M_j > 0$; to see this, note that $\langle M_jx,x \rangle = |B_j^*x|^2$ and so $\langle M_jx,x \rangle = 0$ implies $x \in \ker \widetilde{B}_j^* = \{0\}$. The injectivity of $B_j^*$, or equivalently the  surjectivity of $B_j$, follows from the  feasibility of the data; This is a standard fact in the theory of \eqref{e:BL}; see, for example, \cite{BCCT}.

For (2), similarly it is clear that $\psi(\mathbf{B}) \in [\mathbf{B}]$ and $\sum_{j=1}^m c_j \psi(\mathbf{B})_j^*\psi(\mathbf{B})_j = I_{n}$ holds, as long as $M > 0$. To see this, note that $\langle Mx,x \rangle = \sum_{j=1}^m c_j |B_jx|^2$ and hence $\langle Mx,x \rangle = 0$ implies $x \in \bigcap_{j=1}^m \ker B_j = \{0\}$. The fact that only the zero vector belongs to the kernel of every $B_j$ is also a standard consequence of the feasibility of the data (see, for example, \cite{BCCT}). 
\end{proof}

We say an $m$-transformation $\mathbf{B}$ is \emph{projection-normalised} if
\begin{equation*}
  B_jB_j^* = I_{n_j} \qquad (j \in [m])      
\end{equation*}
and, for a given $m$-exponent $\mathbf{c}$, introduce the following subclass of projection-normalised $m$-transformations
\[
\mathcal{F}(\mathbf{c}) := \{ \mathbf{B} : \mathrm{BL}(\mathbf{B},\mathbf{c}) < \infty, B_jB_j^* = I_{n_j} \, (j \in [m]) \}.
\]
Obviously $\mathcal{G}(\mathbf{c}) \subseteq \mathcal{F}(\mathbf{c})$ and, interestingly, it was shown by Valdimarsson \cite{Vald_JGA} that geometric data uniquely minimise the Brascamp--Lieb constant amongst all projection-normalised $m$-transformations.
\begin{theorem} \cite{Vald_JGA} \label{t:bestbest}
    For any $\mathbf{B} \in \mathcal{F}(\mathbf{c})$, we have $\mathrm{BL}(\mathbf{B},\mathbf{c}) \geq 1$ and equality holds if and only if $\mathbf{B} \in \mathcal{G}(\mathbf{c})$.
\end{theorem}

With Theorem \ref{t:BCCT_geometricequivalent} in mind, the following result provides a reasonable sense in which the space of extremisable/geometric data can be considered ubiquitous.
\begin{theorem} \label{t:genericity}
Let $\mathbf{B} \in \mathcal{F}(\mathbf{c})$ and $\varepsilon > 0$. Then there exists $\widetilde{\mathbf{B}} \in [\mathbf{B}]$ and $\mathbf{G} \in \mathcal{G}(\mathbf{c})$ such that
\begin{equation*} \label{e:GBLdensity_epsilon}
 \|\widetilde{\mathbf{B}} - \mathbf{G}\| < \varepsilon.
\end{equation*}
\end{theorem}
The norm $\|\cdot\|$ may be taken to be any norm on the (euclidean) space of $m$-transformations, but in the interests of concreteness, we fix such a norm from this point.

\begin{remark}
The most naive formulation of density of extremisable data would say that given any $\mathbf{B} \in \mathcal{F}(\mathbf{c})$ and $\varepsilon > 0$, there exists an extremisable data $(\mathbf{X},\mathbf{c})$ such that $\|\mathbf{B} - \mathbf{X}\| < \varepsilon$. However, one can check through simple examples that such a statement cannot hold. For example, let $m = 3$, $B_jx := x \cdot u_j$ for $x \in \mathbb{R}^2$, where $u_j \in \mathbb{R}^2$ are unit vectors for which any pair are linearly independent, and $\mathbf{c} = (1,\frac{1}{2},\frac{1}{2})$. Then $\mathbf{B} \in \mathcal{F}(\mathbf{c})$, but if $\varepsilon > 0$ is sufficiently small it is not possible to find extremisable $(\mathbf{X},\mathbf{c})$ for which $\|\mathbf{B} - \mathbf{X}\| < \varepsilon$. Indeed, for $(\mathbf{X},\mathbf{c})$ to be extremisable when $c_1 = 1$, if we write $X_jx := x \cdot v_j$, then one needs $v_2$ and $v_3$ to be parallel\footnote{This follows, for example, by \cite[Theorem 7.13]{BCCT}.} and this places a uniform positive lower bound on $\|\mathbf{B} - \mathbf{X}\|$. This simple example is informative in the following sense: to recover a density result one should (isotropically) rescale the vectors $u_2$ and $u_3$ to be sufficiently close to the origin (depending on $\varepsilon$), and then one may find $v_j$, with $v_2$ and $v_3$ parallel, for which $\|\mathbf{B} - \mathbf{X}\| < \varepsilon$. Rescaling the $u_j$ means, of course, we are passing from the original data $\mathbf{B}$ to some equivalent data $\widetilde{\mathbf{B}}$.
\end{remark}

As will become clear as we proceed, our proof of Theorem \ref{t:genericity} relies substantially on ideas of Garg, Gurvits, Oliveira and Wigderson in \cite{GGOW}. A suitable $m$-transformation $\widetilde{\mathbf{B}}$ will arise from an explicitly defined sequence of data $(\Phi_k(\mathbf{B}))_k$ determined by the \emph{BL scaling algorithm} in \cite{GGOW}. In this paper, we refer to $(\Phi_k(\mathbf{B}))_k$ as the \emph{BL scaling flow}.

\subsection{Definition of the BL scaling flow}
The BL scaling flow is a discrete flow which is built on repeated iteration of a mapping $\Phi$ on $\mathcal{F}(\mathbf{c})$ and can be found in \cite[Algorithm 1]{GGOW}. Although the algorithm is very powerful, the idea behind it is wonderfully simple: Geometric data are defined by two conditions, \eqref{e:projection} and \eqref{e:isotropy}, and we have already seen in Proposition \ref{p:normalisation} that any feasible data can be mapped to an equivalent data satisfying \emph{either} \eqref{e:projection} or \eqref{e:isotropy} by use of $\phi$ and $\psi$, respectively. The BL scaling flow runs by alternately applying $\phi$ and $\psi$ to the original $m$-transformation in $\mathcal{F}(\mathbf{c})$ in the hope that successively applying these transformations will, in a limiting sense, make \emph{both} \eqref{e:projection} or \eqref{e:isotropy} hold. As the authors of \cite{GGOW} explain, the idea of such a greedy procedure/alternating minimisation algorithm can be found in classical work of Sinkhorn \cite{Sinkhorn}, in which it is used to approximate doubly stochastic matrices by applying procedures which alternately normalise the column sums/row sums, and Gurvits' operator scaling algorithm \cite{Gurvits}. We refer the reader to the survey article by Garg and Oliveira \cite{GargOliveira} for further background and wider perspectives.

Associated with a given initial $m$-transformation $\mathbf{B} \in \mathcal{F}(\mathbf{c})$, we set 
\begin{equation*}
    \Phi(\mathbf{B}) := \phi(\psi(\mathbf{B})).
\end{equation*}
Then $\Phi_k(\mathbf{B})$ is the image of $\mathbf{B}$ under $k$ applications of $\Phi$; i.e.
\[
\Phi_0(\mathbf{B}) = \mathbf{B}, \qquad \Phi_k(\mathbf{B}) = \Phi(\Phi_{k-1}(\mathbf{B})) \qquad (k \in \mathbb{N}).
\]
Observe that from \eqref{e:equivdataBLconstants} we have
\begin{align*}
\mathrm{BL}(\Phi(\mathbf{B}),\mathbf{c}) & = \mathrm{BL}(\psi(\mathbf{B}),\mathbf{c}) \prod_{j=1}^m (\det \psi(\mathbf{B})_j\psi(\mathbf{B})_j^*)^{c_j/2} \\
& = \mathrm{BL}(\mathbf{B},\mathbf{c}) \det\bigg( \sum_{j = 1}^m c_j B_j^*B_j\bigg)^{1/2} \prod_{j=1}^m (\det \psi(\mathbf{B})_j\psi(\mathbf{B})_j^*)^{c_j/2}.
\end{align*}
Some appealing properties of $\Phi$ are recorded below.
\begin{proposition} \label{p:phiprops} 
We have the following.
\begin{enumerate}
\item[(1)] For any $\mathbf{B} \in \mathcal{F}(\mathbf{c})$, we have $\Phi(\mathbf{B}) \in [\mathbf{B}]$.
\item[(2)] $\Phi(\mathcal{F}(\mathbf{c})) \subseteq \mathcal{F}(\mathbf{c})$.
\item[(3)] For any $\mathbf{G} \in \mathcal{G}(\mathbf{c})$, we have $\Phi(\mathbf{G}) = \mathbf{G}$.
\item[(4)] For any $\mathbf{B} \in \mathcal{F}(\mathbf{c})$, we have
\[
\mathrm{BL}(\Phi(\mathbf{B}),\mathbf{c}) \leq \mathrm{BL}(\mathbf{B},\mathbf{c}) \det\bigg( \sum_{j = 1}^m c_j B_j^*B_j\bigg)^{1/2}
\]
\item[(5)] For any $\mathbf{B} \in \mathcal{F}(\mathbf{c})$, we have
\[
\mathrm{BL}(\Phi(\mathbf{B}),\mathbf{c}) \leq \mathrm{BL}(\mathbf{B},\mathbf{c}).
\]
\end{enumerate}
\end{proposition}
Properties (1)--(3) are clear from Proposition \ref{p:normalisation}. Properties (4) and (5) follow immediately from the subsequent lemmata. 
\begin{lemma} \cite{GGOW} \label{l:isodec}
    Suppose $(\mathbf{B},\mathbf{c})$ is feasible and $B_jB_j^* = I_{n_j}$ for all $j \in [m]$. Then 
    \[
    \det\bigg(\sum_{j = 1}^m c_j B_j^*B_j\bigg) \leq 1.
    \]
\end{lemma}

\begin{lemma} \cite{GGOW} \label{l:projdec}
    Suppose $(\mathbf{B},\mathbf{c})$ is feasible and $\sum_{j = 1}^m c_j B_j^*B_j = I_n$. Then 
    \[
    \prod_{j=1}^m (\det B_jB_j^*)^{c_j} \leq 1.
    \]
\end{lemma}
We include the proofs of these lemmata in an Appendix, following exactly the arguments in \cite{GGOW}. We include the arguments here to help make the presentation self-contained and allows us to emphasise the delightful simplicity of the arguments in \cite{GGOW}.

\subsection{Convergence of the BL scaling flow}
The main result here is the following, which is clearly sufficient to prove Theorem \ref{t:genericity}.
\begin{theorem} \label{t:GBLdensity}
For any $\mathbf{B} \in \mathcal{F}(\mathbf{c})$ there exist $\mathbf{G} \in \mathcal{G}(\mathbf{c})$ and a subsequence of the BL scaling flow $(\Phi_{\ell_k}(\mathbf{B}))_{k \geq 1}$ such that
\begin{equation} \label{e:GBLdensity}
\lim_{k \to \infty} \|\Phi_{\ell_k}(\mathbf{B}) - \mathbf{G}\| = 0.
\end{equation}
\end{theorem}
To prove this, it suffices to prove the following.
\begin{proposition} \label{p:gconv}
For any $\mathbf{B} \in \mathcal{F}(\mathbf{c})$, we have
\[
\lim_{k \to \infty} \mathrm{tr} \, \bigg( \sum_{j=1}^m c_j\Phi_k(\mathbf{B})_j^*\Phi_k(\mathbf{B})_j - I_n   \bigg)^2 = 0.
\]
\end{proposition}
Assuming Proposition \ref{p:gconv} for the moment, using the compactness of the set of all projection-normalised $m$-transformations, we are able to extract a subsequence $(\Phi_{\ell_k}(\mathbf{B}))_{k \geq 1}$ and some projection-normalised $\mathbf{G}$ such that \eqref{e:GBLdensity} holds. But, by continuity, Proposition \ref{p:gconv} implies that $\sum_{j=1}^m c_jG_j^*G_j = I_n$; that is, $\mathbf{G} \in \mathcal{G}(\mathbf{c})$ and thus we have deduced Theorem \ref{t:GBLdensity}.

The proof of Proposition \ref{p:gconv} rests on the following generalisation of Lemma \ref{l:isodec}.
\begin{lemma} \label{l:isodecstable}
    Suppose $A : \mathbb{R}^n \to \mathbb{R}^n$ is a positive definite transformation and $\textrm{tr} \, A = n$. If $\varepsilon \in (0,1)$ and $\textrm{tr} \, (A - I_n)^2 \geq  \varepsilon$, then
    \[
    \det A \leq e^{-\varepsilon/6}.
    \]
\end{lemma}
\begin{proof}[Proof of Lemma \ref{l:isodecstable}]
This is \cite[Lemma 10.2]{GGOW} and the argument is just the one given in the Appendix for Lemma \ref{l:isodec} except one invokes a stable version of the arithmetic-geometric mean inequality; see \cite[Lemma 3.3]{LSW}.
\end{proof}

\begin{proof}[Proof of Proposition \ref{p:gconv}]
Arguing by contradiction, we suppose that the claimed convergence fails to hold, in which case we have the existence of some $\varepsilon_0 > 0$ such that for any $k \in \mathbb{N}$ we have
\begin{equation}
\mathrm{tr} \, \bigg( \sum_{j=1}^m c_j\Phi_{i_k}(\mathbf{B})_j^*\Phi_{i_k}(\mathbf{B})_j - I_n   \bigg)^2 \geq \varepsilon_0
\end{equation}
for some $i_k \in \mathbb{N}$ satisfying $i_{k + 1} \geq i_k + 1$. It follows from this and Proposition \ref{p:phiprops}(4) that
\begin{align*}
\mathrm{BL}(\Phi_{i_{k}+1}(\mathbf{B}),\mathbf{c}) & \leq \mathrm{BL}(\Phi_{i_{k}}(\mathbf{B}),\mathbf{c}) \det\bigg( \sum_{j = 1}^m c_j \Phi_{i_{k}}(\mathbf{B})_j^*\Phi_{i_{k}}(\mathbf{B})_j\bigg)^{1/2} \\
& \leq e^{-\varepsilon_0/12}\mathrm{BL}(\Phi_{i_{k}}(\mathbf{B}),\mathbf{c}).
\end{align*}
Since $i_{k} \geq i_{k-1} + 1$, from Proposition \ref{p:phiprops}(5) we may conclude that
\[
\mathrm{BL}(\Phi_{i_{k}+1}(\mathbf{B}),\mathbf{c}) \leq e^{-\varepsilon_0/12}\mathrm{BL}(\Phi_{i_{k-1} + 1}(\mathbf{B}),\mathbf{c}).
\]
By iterating this bound
\[
\mathrm{BL}(\Phi_{i_{k}+1}(\mathbf{B}),\mathbf{c}) \leq e^{-k\varepsilon_0/12}\mathrm{BL}(\Phi_{i_{0} + 1}(\mathbf{B}),\mathbf{c})
\]
for any $k \in \mathbb{N}$. On the other hand, $\Phi_{i_{k}+1}(\mathbf{B}) \in \mathcal{F}(\mathbf{c})$ and hence, by Theorem \ref{t:bestbest} we get
\[
e^{k\varepsilon_0/12} \leq \mathrm{BL}(\Phi_{i_{0} + 1}(\mathbf{B}),\mathbf{c})
\]
for any $k \in \mathbb{N}$. Since we also know $\Phi_{i_{0} + 1}(\mathbf{B}) \in \mathcal{F}(\mathbf{c})$,  the right-hand side of the above inequality is finite and we get a contradiction by taking $k$ sufficiently large.
\end{proof}

\begin{remark}
Proposition \ref{p:gconv} heavily relies on the argument leading to \cite[Theorem 10.3]{GGOW}. The difference arises as a result of the contrasting perspectives between \cite{GGOW} and our own; \cite{GGOW} is primarily focused on establishing various results of a \emph{quantitative} nature regarding the Brascamp--Lieb constant. This includes obtaining explicit upper bounds on the Brascamp--Lieb constant and a polynomial time algorithm for arbitrarily close approximations to the Brascamp--Lieb constant. For this, the authors of \cite{GGOW} devised the BL scaling flow that we have used here, but in order to make everything quantitative, the Brascamp--Lieb data are assumed to be rational in \cite{GGOW}. For example, \cite[Theorem 1.5]{GGOW} is very closely related to Proposition \ref{p:gconv} but is concerned with Brascamp--Lieb data $(\mathbf{B},\mathbf{c})$ where all entries in the matrices defining the $B_j$ are rational numbers and the exponents $c_j$ are rational. Also, Proposition \ref{p:gconv} is very close to \cite[Theorem 10.3]{GGOW}, but the latter is stated for Brascamp--Lieb data $(\mathbf{B},\mathbf{c})$ where all entries in the matrices defining the $B_j$ are rational numbers. Section \ref{section:ubiquity} of the present paper should therefore be viewed as providing clarification that the key ideas of \cite{GGOW} can be used without any assumption on the rationality of the Brascamp--Lieb data, as long as one is prepared to achieve only qualitative results. We also hope that this brings into the spotlight the elegant arguments from \cite{GGOW}, particularly as it seems likely that they will find further applications in the theory of Brascamp--Lieb-type inequalities.
\end{remark}

For the application to the adjoint Brascamp--Lieb inequality in Section \ref{section:ABL} below, the convergence of the subsequence in Theorem \ref{t:GBLdensity} is sufficient. Nevertheless, it seems to be an interesting question as to whether the full BL scaling flow converges\footnote{This is related to the discussion in Sections 3.1-3.2 in \cite{KLLR_STOC}.}. 
\begin{question}
 For any $\mathbf{B} \in \mathcal{F}(\mathbf{c})$ does there exist $\mathbf{G} \in \mathcal{G}(\mathbf{c})$ such that
\begin{equation} \label{e:GBLdensityGOAL}
\lim_{k \to \infty} \|\Phi_{k}(\mathbf{B}) - \mathbf{G}\| = 0?
\end{equation}   
\end{question}

Although we have been unable to establish/disprove \eqref{e:GBLdensityGOAL}, we may deduce that the Brascamp--Lieb constant converges to $1$ along the full BL scaling flow as a consequence of Theorem \ref{t:GBLdensity}. For this we also use the following continuity property of the Brascamp--Lieb constant.
\begin{theorem} \cite{BBCF} \label{t:continuous}
    For a fixed $m$-exponent $\mathbf{c}$, the mapping $\mathbf{L} \mapsto \mathrm{BL}(\mathbf{L},\mathbf{c})$ is continuous.
\end{theorem}

\begin{corollary} \label{c:BLconstantconvergence}
For any $\mathbf{B} \in \mathcal{F}(\mathbf{c})$, we have
\begin{equation} \label{e:GBLconstantconv}
\lim_{k \to \infty} \mathrm{BL}(\Phi_k(\mathbf{B}),\mathbf{c}) = 1.
\end{equation}
\end{corollary}
\begin{proof}
From Proposition \ref{p:phiprops}(5) we know that $(\mathrm{BL}(\Phi_k(\mathbf{B}),\mathbf{c}))_{k \geq 1}$ is monotone non-decreasing and hence has a limit. However, from Theorems \ref{t:GBL}, \ref{t:GBLdensity}, and \ref{t:continuous}, we obtain
\begin{equation*}
\lim_{k \to \infty} \mathrm{BL}(\Phi_{\ell_k}(\mathbf{B}),\mathbf{c}) = \mathrm{BL}(\mathbf{G},\mathbf{c}) = 1.
\end{equation*}
Hence \eqref{e:GBLconstantconv} follows.
\end{proof}

\section{An application to the adjoint Brascamp--Lieb inequality} \label{section:ABL}

The adjoint Brascamp--Lieb inequality was recently discovered by Bennett--Tao \cite{BennettTao} and takes the form
\begin{equation} \label{e:ABL}
   \|f\|_{L^p(\mathbb{R}^n)} \leq \mathrm{ABL}(\mathbf{B},\mathbf{c},\theta,p) \prod_{j=1}^m \| (B_j)_*f\|_{L^{p_j}(\mathbb{R}^{n_j})}^{\theta_j}. 
\end{equation}
Here, the $B_j : \mathbb{R}^n \to \mathbb{R}^{n_j}$ are surjective linear transformations, $c_j > 0$, $\theta_j \in (0,1]$ are such that $\sum_{j=1}^m \theta_j = 1$, and $p \in (0,1]$. Also, the $p_j \in (0,1]$ are determined by
\begin{equation} \label{e:ABLexponentsrelation}
    c_j\bigg(\frac{1}{p} - 1 \bigg) = \theta_j\bigg(\frac{1}{p_j} - 1 \bigg).
\end{equation}
The constant $\mathrm{ABL}(\mathbf{B},\mathbf{c},\theta,p)$ is the best constant such that \eqref{e:ABL} holds for all non-negative functions $f : \mathbb{R}^n \to \mathbb{R}$. Also, $(B_j)_*f$ denotes the push-forward of $f$ by $B_j$ determined by
\[
\int_{\mathbb{R}^{n_j}} (B_j)_*f(y)F(y) \, dy = \int_{\mathbb{R}^n} f(x) F(B_jx) \, dx
\]
for all non-negative measurable functions $F$ on $\mathbb{R}^{n_j}$. Associated with the above parameters, let
\[
C(\mathbf{c},\theta,\mathbf{n},p) := p^{-\frac{n}{2p}} \prod_{j = 1}^m p_j^{\frac{\theta_jn_j}{2p_j}}.
\]

In this setting, as an analogue of \eqref{e:equivdataBLconstants} for $\mathrm{ABL}(\mathbf{B},\mathbf{c},\theta,p)$, if $\widetilde{\mathbf{B}} \in [\mathbf{B}]$ with intertwining transformations $T_j : \mathbb{R}^{n_j} \to \mathbb{R}^{n_j}$ and $T : \mathbb{R}^n \to \mathbb{R}^n$, then we have
\begin{equation} \label{e:equivdataABLconstants}
    \mathrm{ABL}(\widetilde{\mathbf{B}},\mathbf{c},\theta,p) = \frac{\prod_{j=1}^m |\det T_j|^{c_j(\frac{1}{p}-1)}}{|\det T|^{\frac{1}{p} - 1}} \mathrm{ABL}(\mathbf{B},\mathbf{c},\theta,p).
\end{equation}
To see this, one can easily check by various changes of variables that
\begin{equation} \label{e:pushforwardequiv}
(\widetilde{B}_j)_*(f \circ T)(y) = \frac{|\det T_j|}{|\det T|} (B_j)_*f(T_jy)  .
\end{equation}
Therefore, by making use of \eqref{e:ABLexponentsrelation}, we obtain
\begin{equation*}
    \frac{\|f \circ T\|_{L^p(\mathbb{R}^n)}}{\prod_{j=1}^m \| (\widetilde{B}_j)_*(f \circ T)\|_{L^{p_j}(\mathbb{R}^{n_j})}^{\theta_j}}
    = \frac{\prod_{j=1}^m |\det T_j|^{c_j(\frac{1}{p}-1)}}{|\det T|^{\frac{1}{p}-1}} \frac{\|f\|_{L^p(\mathbb{R}^n)}}{\prod_{j=1}^m \| (B_j)_*f\|_{L^{p_j}(\mathbb{R}^{n_j})}^{\theta_j}} 
\end{equation*}
and hence \eqref{e:equivdataABLconstants} follows.

A fundamental observation made by  Bennett and Tao in \cite{BennettTao} is that the adjoint Brascamp--Lieb inequality is equivalent to the original Brascamp--Lieb inequality in the following sense. 
\begin{theorem} \cite[Theorem 1.11]{BennettTao} \label{t:BennettTao}
We have
\begin{equation} \label{e:ABLgoal}
C(\mathbf{c},\theta,\mathbf{n},p)\mathrm{BL}(\mathbf{B},\mathbf{c})^{\frac{1}{p}-1} \leq \mathrm{ABL}(\mathbf{B},\mathbf{c},\theta,p)  \leq \mathrm{BL}(\mathbf{B},\mathbf{c})^{\frac{1}{p}-1}.
\end{equation}
\end{theorem}
The upper bound
\[
\mathrm{ABL}(\mathbf{B},\mathbf{c},\theta,p) \leq \mathrm{BL}(\mathbf{B},\mathbf{c})^{\frac{1}{p}-1}
\]
was proved in \cite[Section 3]{BennettTao} by clever use of H\"older's inequality. In the remainder of this section, we show how Theorem \ref{t:GBLdensity} can be used to quickly establish the remaining inequality in \eqref{e:ABLgoal} for all feasible data.

Let $C := C(\mathbf{c},\theta,\mathbf{n},p)$. Thanks to \eqref{e:equivdataBLconstants} and \eqref{e:equivdataABLconstants}, it suffices to consider $\mathbf{B} \in \mathcal{F}(\mathbf{c})$. Also, it suffices to prove
\begin{equation} \label{e:ABLgoalclose}
(C - \varepsilon)\mathrm{BL}(\mathbf{B},\mathbf{c})^{\frac{1}{p}-1} < \mathrm{ABL}(\mathbf{B},\mathbf{c},\theta,p)
\end{equation}
for any $\varepsilon > 0$. Suppose, for a contradiction, that there exists some $\varepsilon > 0$ for which 
\[
(C - \varepsilon)\mathrm{BL}(\mathbf{B},\mathbf{c})^{\frac{1}{p}-1} \geq \mathrm{ABL}(\mathbf{B},\mathbf{c},\theta,p).
\] 
Using \eqref{e:equivdataBLconstants} and \eqref{e:equivdataABLconstants} again, this implies
\[
(C - \varepsilon)\mathrm{BL}(\Phi_{\ell_k}(\mathbf{B}),\mathbf{c})^{\frac{1}{p}-1} \geq \mathrm{ABL}(\Phi_{\ell_k}(\mathbf{B}),\mathbf{c},\theta,p)
\] 
for every $k \in \mathbb{N}$. By Corollary \ref{c:BLconstantconvergence}, we have
\begin{equation*}
    C - \varepsilon = \liminf_{k \to \infty} \left((C - \varepsilon)\mathrm{BL}(\Phi_{\ell_k}(\mathbf{B}),\mathbf{c})^{\frac{1}{p}-1}\right). 
\end{equation*}
On the other hand, if we follow \cite{BennettTao} and define $\mathrm{ABL_g}(\mathbf{B},\mathbf{c},\theta,p)$ to be the best constant such that \eqref{e:ABL} holds for gaussian functions of the form
\begin{equation} \label{e:gaussian}
f(x) = e^{-\pi\langle Ax,x \rangle}
\end{equation}
over all choices of positive definite transformations $A : \mathbb{R}^n \to \mathbb{R}^n$, then we have that $\mathbf{L} \mapsto \mathrm{ABL_g}(\mathbf{L},\mathbf{c},\theta,p)$ is lower semi-continuous\footnote{Since it is the supremum of a family of continuous functions.}. Therefore
\begin{align*}
    \liminf_{k \to \infty}\mathrm{ABL}(\Phi_{\ell_k}(\mathbf{B}),\mathbf{c},\theta,p) & \geq \liminf_{k \to \infty}\mathrm{ABL_g}(\Phi_{\ell_k}(\mathbf{B}),\mathbf{c},\theta,p) \\
    & \geq \mathrm{ABL_g}(\mathbf{G},\mathbf{c},\theta,p).
\end{align*}
Finally, by considering the isotropic gaussian \eqref{e:gaussian} with  $A = I_n$, one sees that 
\[
\mathrm{ABL_g}(\mathbf{G},\mathbf{c},\theta,p) \geq C,
\]
which gives the desired contradiction.

\section*{Appendix}

\begin{proof}[Proof of Lemma \ref{l:isodec}]
As observed in \cite{GGOW} (in the course of proving \cite[Theorem 9.2]{GGOW}), this follows from the arithmetic-geometric mean inequality: If we write $\lambda_1,\ldots,\lambda_n$ for the eigenvalues of $M := \sum_{j = 1}^m c_j B_j^*B_j$, then
\[
(\det M)^{1/n} = \prod_{k = 1}^n \lambda_k^{1/n} \leq \frac{1}{n}\sum_{k=1}^n \lambda_k = \frac{1}{n} \mathrm{tr} \, M = 1.
\]
The last step follows since our assumption gives $\mathrm{tr} \, (B_j^*B_j) = \mathrm{tr} \, (B_jB_j^*) = n_j$, and the feasibility of the data implies $n = \sum_{j = 1}^m c_jn_j$.
\end{proof}

\begin{proof}[Proof of Lemma \ref{l:projdec}]
This also follows from the proof of \cite[Theorem 9.2]{GGOW} and we repeat their argument here. Writing $\lambda_{j,1},\ldots,\lambda_{j,n_j}$ for the eigenvalues of $B_jB_j^*$, then concavity of the logarithm implies
\begin{align*}
\log \prod_{j=1}^m (\det B_jB_j^*)^{c_j} & = \sum_{j=1}^m n_jc_j \frac{1}{n_j}\sum_{k=1}^{n_j} \log \lambda_{j,k} \\
& \leq \sum_{j=1}^m n_jc_j \log \bigg(\frac{1}{n_j}\sum_{k=1}^{n_j}  \lambda_{j,k}\bigg) \\
& = \sum_{j=1}^m n_jc_j \log \bigg(\frac{1}{n_j} \mathrm{tr} \, (B_jB_j^*)  \bigg).
\end{align*}
The feasibility of the data implies $n = \sum_{j = 1}^m c_jn_j$ (by a standard scaling argument), so a further use of the concavity of the logarithm yields
\begin{align*}
\log \prod_{j=1}^m (\det B_jB_j^*)^{c_j} & \leq n\sum_{j=1}^m \frac{n_jc_j}{n} \log \bigg(\frac{1}{n_j} \mathrm{tr} \, (B_jB_j^*)  \bigg) \\
& \leq n\log  \sum_{j=1}^m \frac{c_j}{n} \mathrm{tr} \, (B_jB_j^*)  \\
& = 0.
\end{align*}
The final step follows since $\sum_{j = 1}^m c_j B_j^*B_j = I_n$.
\end{proof}

\begin{remark}
As pointed out in \cite[Section 9]{GGOW}, by Theorem \ref{t:Lieb},
\begin{align*}
\mathrm{BL}(\mathbf{B},\mathbf{c})^2 & = \sup_{A_j > 0} \frac{\prod_{j=1}^m (\det A_j)^{c_j}}{\det(\sum_{j=1}^m c_jB_j^*A_jB_j)} \\
& \geq \frac{1}{\det(\sum_{j=1}^m c_jB_j^*B_j)} \geq 1
\end{align*}
holds whenever $\mathbf{B} \in \mathcal{F}(\mathbf{c})$, where the last bound follows from Lemma \ref{l:isodec}. This gives a proof (much shorter than the original) of the lower bound in Theorem \ref{t:bestbest}. The characterisation of minimising data can be extracted from an inspection of the proof of Lemma \ref{l:isodec}. Indeed, if $\mathrm{BL}(\mathbf{B},\mathbf{c}) = 1$ then all the eigenvalues of $M$ must be equal (the equality case in the arithmetic-geometric mean inequality). The fact that $\mathbf{B} \in \mathcal{F}(\mathbf{c})$ means $\mathrm{tr} \, M = n$, and so all the eigenvalues must in fact be equal to one. In other words, $M = I_n$ and thus $\mathbf{B} \in \mathcal{G}(\mathbf{c})$.
\end{remark}

\end{document}